%% file: Main.tex
\documentclass[a4paper,reqno]{amsart}
\usepackage{amsmath}
\usepackage{amsfonts}
\usepackage{amsthm}
\usepackage{amssymb}
\usepackage{mathrsfs}
\usepackage{bm}
\usepackage{booktabs}
\usepackage{lipsum}
\usepackage{bussproofs}
\usepackage{hyperref}
\usepackage{tikz}
\usetikzlibrary{arrows,automata,positioning,decorations.markings}
\usepackage{python}
\usepackage{lineno}

\input{Definitions.tex}

\newtheorem{presentment}{}[section]

\theoremstyle{plain}
\newtheorem{theorem}[presentment]{Theorem}
\newtheorem{proposition}[presentment]{Proposition}
\newtheorem{lemma}[presentment]{Lemma}

\newtheorem{observation}[presentment]{Observation}

\theoremstyle{definition}
\newtheorem{definition}[presentment]{Definition}
\newtheorem{example}[presentment]{Example}

\newtheorem{question}[presentment]{Question}

\theoremstyle{remark}

\newtheorem*{note}{Note}
\newtheorem*{acknowledgements}{Acknowledgements}

    \def\fraction#1/#2{\tfrac{#1}{#2}}
    \def\Fraction#1/#2{\frac{#1}{#2}}
\newcommand*{\BigFraction}{}
    \def\BigFraction#1/#2{\dfrac{\displaystyle{#1}}{\displaystyle{#2}}}

    \def\setOf#1:#2{\set{{#1} \,\left\vert\vphantom{{#1}{#2}}\right. {#2}}}

    \def\sequenceOf#1:#2{\sequence{{#1} \,\left\vert\vphantom{{#1}{#2}}\right. {#2}}}

    \def\function#1:#2->#3{{{#1} \colon {#2} \to {#3}}}



\usepackage{pst-node}
\usepackage{tikz-cd} 
\tikzset{node distance=2cm, auto}
\allowdisplaybreaks[0]
\title{Shadowing in CR-Dynamical Systems}

\author{Andrew Wood}

\thanks{Affiliation: Mathematical Institute, University of Oxford,
Andrew Wiles Building, Radcliffe Observatory Quarter, Woodstock Road,
Oxford OX2 6GG, United Kingdom.
Corresponding author: Andrew Wood,
\href{mailto:andrew.wood@stcatz.ox.ac.uk}
{\texttt{andrew.wood@stcatz.ox.ac.uk}}.
ORCID:
\href{https://orcid.org/0009-0009-9948-2102}
{\texttt{0009-0009-9948-2102}}.}


\begin{document}

\maketitle

%---------------------------------------------------------------
\begin{abstract}
A CR-dynamical system is a pair $\p{X,G}$, where $X$ is a compact
metric space and $G$ is a closed relation on $X$. We introduce four
$\p{i,j}$-shadowing properties, obtained by independently varying the
quantifiers in the definitions of pseudo-orbit and shadowing. These
properties coincide with the classical shadowing property for
single-valued systems, while $\p{2,2}$-shadowing agrees with the
established notion of shadowing for set-valued dynamical systems. We
investigate how the classical characterisations of shadowing for the
identity map, periodic maps and isometries extend to closed relations.
In particular, we obtain rigidity results for $\p{2,1}$-shadowing and
positive results for $\p{1,2}$-shadowing on totally disconnected
spaces. Finally, whenever the infinite Mahavier product is non-empty, 
we prove that a closed relation has the
$\p{2,2}$-shadowing property if, and only if, its associated Mahavier
dynamical system has the shadowing property. This removes the
surjectivity hypothesis from a recent corresponding result for
set-valued dynamical systems.
\end{abstract}

\makeatletter
\renewcommand\@makefnmark{}
\makeatother
%\footnotetext{Supported by the Marsden Fund Council from Government funding, administered by the Royal Society of New Zealand. This work appears in the MSc thesis of the author.} 

%Z\linenumbers
%---------------------------------------------------------------
\section{Introduction}

\renewcommand{\thepresentment}{\Alph{presentment}}

An important concept in classical dynamics is the shadowing property, introduced independently by Anosov~\cite{anosov} and Bowen~\cite{bowen}, with its inception coming from computers.  As such, topological dynamical systems with the shadowing property are suited for computer simulations, yielding fruitful applications~\cite{KIRCHGRABER2000897, Urs_Kirchgraber_2004, STOFFER2001415, Stoffer1999RigorousVO}.  In classical dynamics, we wish to calculate trajectories of points within our space, but computers make computational errors. This derives the notion of pseudo-orbits which have an error term.  The shadowing property then aims to ask the question: given a pseudo-orbit, can we find an actual trajectory with close distance to the pseudo-orbit? Shadowing interacts nicely with other important dynamical properties. In a topological dynamical system with the shadowing property, many dynamical properties are equivalent~\cite{ ASKRI2023108663, KWIETNIAK201219, meddaugh_shadowing_recurrence}, allowing for richer dynamics.  Standard references for shadowing include Palmer~\cite{palmer} and Pilyugin~\cite{pilyugin_book}.

Beyond topological dynamics, the shadowing property is widely studied in  hyperbolic dynamics, with Anosov's Shadowing Theorem~\cite[Corollary $5.3.2$]{brin_stuck_dynamical_systems} (also referred to as the Shadowing Lemma) being an important result.  Due to the importance of such a property, it is no surprise shadowing has been generalised to set-valued dynamical systems~\cite{Luo2020,  yin2024finitetypefundamentalobjects, specification, set_valued_weak_forms_shadowing, Zhou2022, yin_shadowing_property_set_valued_map, yu_rieger, zhao_wang}.

Banič, Erceg, Rogina, and Kennedy's paper ``Minimal dynamical systems with closed relations''~\cite{minimality_CR} formally introduces CR-dynamical systems $\p{X, G}$, where $X$ is a compact metric space and $G$ is a closed relation on $X$. Recently, this has led to the further study of CR-dynamical systems~\cite{nagar_transitivity_CR, transitivity_CR, greenwood2025transitivitycrdynamicalsystems}, generalising transitivity and transitive points from topological dynamical systems to CR-dynamical systems.  Studying the dynamics of closed relations and set-valued functions has proved useful, including in applied areas such as economics. The Christiano-Harrison model in macroeconomics~\cite{economics} is one such example.

In this paper, we introduce shadowing for CR-dynamical systems. The closed-relation and set-valued viewpoints are closely related. Indeed, by~\cite[Theorem $1.2$]{ingram}, a set-valued function $F\colon X\to 2^X$ is upper semi-continuous if, and only if, its graph is closed in $X\times X$. Consequently, when $\domain{G}=X$, the relation $G$ is precisely the graph of a compact-valued upper
semi-continuous function, and hence gives a set-valued dynamical system. Allowing $\domain{G}\neq X$ amounts to permitting partial dynamics, in which points may have no successor at all. We use closed relations not to distinguish full-domain CR-dynamical systems from set-valued dynamical systems, but to accommodate partial dynamics, inverse relations and relation composition naturally.

We introduce four $\p{i,j}$-shadowing properties by independently varying the quantifiers in the definitions of pseudo-orbit and shadowing. The $\p{2,2}$-shadowing property is precisely the shadowing property already studied for set-valued dynamical systems. All four notions coincide with the classical shadowing property when $G$ is the graph of a single-valued function.

It is known that the identity map on a compact metric space has the shadowing property if, and only if, the space is totally disconnected~\cite[Theorem~3]{meddaugh_shadowing_recurrence}. The same characterisation holds for periodic maps~\cite[Theorem~5]{meddaugh_shadowing_recurrence} and, more generally, for equicontinuous surjections~\cite[Theorem~4]{moothathu2011implications}; in particular, it holds for isometries.

The main purpose of this paper is to understand how the classical characterisations of shadowing for the identity map, periodic maps, and isometries change when a self-map is replaced by a closed relation. The quantifier choices are useful for different reasons:  the $\p{2, 1}$-shadowing property is rigid enough to detect branching, the $\p{1, 2}$-shadowing property gives natural positive generalisations, and the $\p{2, 2}$-shadowing property is the established notion from set-valued dynamics. Our first main theorem generalises the identity-map characterisation. 

\begin{theorem}\label{theorem:main-A}
Let $\p{X, G}$ be a CR-dynamical system with $\Delta_X \subseteq G$. 
\begin{enumerate}
\item If $\p{X, G}$ has the $\p{2, 1}$-shadowing property, then $G = \Delta_X$ and $X$ is totally disconnected. 
\item If $X$ is totally disconnected, then $\p{X, G}$ has the $\p{1, 2}$-shadowing property.  
\end{enumerate}
Both statements are sharp with respect to the positions of the quantifiers.
\end{theorem}

The periodic-map characterisation suggests replacing the hypothesis $\Delta_X \subseteq G$ with the existence of a positive integer $n$ such that $\Delta_X \subseteq G^n$.  The rigidity of the $\p{2, 1}$-shadowing property as in $\p{1}$ of Theorem~\ref{theorem:main-A} survives. For the other direction in the classical periodic-map characterisation, uniform continuity is employed in the proof.  With $G^0 = \Delta_X$, we define set-valued functions $F_k \colon X \to 2^X$ for each $k \in \set{0, \ldots, n}$ by 
\[
F_k\of{x} = \setOf{y \in X}:{y \in G^k\of{x} \text{ and } x \in G^{n-k}\of{y}}
\]
for each $x \in X$. It is straightforward to check that $F_k$ is upper semi-continuous for each $k$.   We call $F_k$ uniformly upper semi-continuous (see~\cite{set_valued_uniform_continuity}) if, for every $\epsilon > 0$, there exists a $\delta > 0$ such that for all $x, y \in X$, $d\of{x, y} \leq \delta$  implies
\[
F_k\of{x} \subseteq \Union_{z \in F_k\of{y}} B\of{z, \epsilon}.  
\]
We now give our second main theorem generalising the periodic-map characterisation.  

\begin{theorem}\label{theorem:main-B}
Let $\p{X, G}$ be a CR-dynamical system, such that there exists positive integer $n$ with $\Delta_X \subseteq G^n$, and let $F_k$ be as above.  
\begin{enumerate}
\item If $\p{X, G}$ has the $\p{2, 1}$-shadowing property, then $G = \Graph{f}$ for a periodic homeomorphism $f \colon X \to X$ with $f^n = \emph{id}_X$, and $X$ is totally disconnected. 
\item If $X$ is totally disconnected and $F_k$ is uniformly upper semi-continuous for each $0 \leq k \leq n$, then $\p{X, G}$ has the $\p{1, 2}$-shadowing property. 
\end{enumerate}
\end{theorem}

For the graph $\Graph{f}$ of a periodic map $f$, we have $F_k\of{x}=\set{f^k\of{x}}$. Thus, uniform upper
semi-continuity of each $F_k$ follows from uniform continuity of the iterates of $f$.  Theorem~\ref{theorem:main-B} recovers both directions of the classical periodic-map theorem.  We also give an example in which $\p{1,2}$-shadowing holds although one of the set-valued functions $F_k$ is not uniformly upper semi-continuous. We now give our third main theorem generalising the isometry characterisation, which has a connection to Mahavier dynamical systems (see~\cite{BANIC2026109742}).    

\begin{theorem}\label{theorem:main-C}
Let $\p{X, G}$ be a CR-dynamical system such that there exists an isometry $f \colon X \to X$ with $\Graph{f} \subseteq G$. Then:
\begin{enumerate}
\item $\p{X, f}$ has the shadowing property if, and only if, $X$ is totally disconnected if, and only if, the shift map on $X_{\Graph{f}}^+$ has the shadowing property;
\item if $\p{X, G}$ has the $\p{2, 1}$-shadowing property, then $G = \Graph{f}$ and $X$ is totally disconnected; 
\item if $X$ is totally disconnected, then $\p{X, G}$ has the $\p{1, 2}$-shadowing property; 
\item $X_G^+$ is totally disconnected if, and only if, $X$ is totally disconnected. 
\end{enumerate}
\end{theorem} 

Very recently, Zhao and Wang~\cite[Corollary $3.4$]{zhao_wang} proved, for a surjective upper semi-continuous set-valued map $F$, that $F$ has shadowing if, and only if, the shift on the generalized inverse limit of its inverse multifunction has shadowing. Our final main theorem removes the surjectivity hypothesis. 

\begin{theorem}\label{theorem:main-D}
Let $\p{X, G}$ be a CR-dynamical system such that $X_G^+ \neq \emptySet$.  The following are equivalent. 
\begin{enumerate}
\item $\p{X, G}$ has the $\p{2, 2}$-shadowing property. 
\item The Mahavier dynamical system $\p{X_G^+, \sigma_G^+}$ has the shadowing property. 
\end{enumerate}
\end{theorem}

The paper is organised as follows. Section~\ref{section:preliminaries} covers preliminaries.  We introduce the $\p{i, j}$-shadowing properties in Section~\ref{section:ij-shadowing}, giving observations, their elementary implication diagram, and address the possible vacuity of the $\p{1, j}$-shadowing properties. The remainder of the paper is dedicated to proving our main theorems. Section~\ref{section:diagonal} proves Theorems~\ref{theorem:main-A} and~\ref{theorem:main-B}, and Section~\ref{section:isometries} proves Theorems~\ref{theorem:main-C} and~\ref{theorem:main-D}.

\renewcommand{\thepresentment}{\thesection.\arabic{presentment}}

%----------------------------------------------------------------
%----------------------------------------------------------------
\section{Preliminaries}\label{section:preliminaries}

Throughout this paper, $\p{X,d}$ denotes a non-empty compact metric
space and $\naturals=\set{0,1,2,\ldots}$. We refer to
\cite{minimality_CR,transitivity_CR} for further background on
CR-dynamical systems and to Palmer~\cite{palmer} and
Pilyugin~\cite{pilyugin_book} for classical shadowing. We recall here
only the notation and constructions used throughout the paper.

We write
\[
2^X
=
\setOf{A\subseteq X}:{A\neq\emptySet\text{ and }A\text{ is closed}}.
\]
Since $X$ is compact, every member of $2^X$ is compact. We use $\Delta_X$ to denote the diagonal in $X \times X$. 

\subsection{Relations, domains and viability}

Let $G$ be a relation on $X$. For each $x\in X$ and $A\subseteq X$,
define
\[
G\of{x}
=
\setOf{y\in X}:{\p{x,y}\in G}
\]
and
\[
G\of{A}
=
\Union_{x\in A}G\of{x}.
\]
The \emph{domain} of $G$ is
\[
\domain{G}
=
\setOf{x\in X}:{G\of{x}\neq\emptySet}.
\]

For each positive integer $n$, the \emph{$n$-th Mahavier product} of $G$ is
\[
\mahavier{n}{G}
=
\setOf{\p{x_0,\ldots,x_n}\in X^{n+1}}:
{\p{x_k,x_{k+1}}\in G
\text{ for every }0\leq k<n}.
\]
We use the convention
\[
\mahavier{0}{G}=X.
\]
The infinite Mahavier product of $G$ is
\[
X_G^+
=
\setOf{\sequenceOf{x_n}:{n\in\naturals}\in X^{\naturals}}:
{\p{x_n,x_{n+1}}\in G
\text{ for every }n\in\naturals}.
\]
Since $G$ is closed, $X_G^+$ is a closed subspace of
$X^{\naturals}$ and is therefore compact.
For each positive integer $n$, define
\[
G^n
=
\setOf{\p{x_0,x_n}\in X\times X}:
{\exists \p{x_0, \ldots, x_n} \in \mahavier{n}{G}}.
\]
We use the convention $G^0=\Delta_X$. If $G$ is closed, then
$\domain{G}$, $G\of{x}$, and $G^n$ are
compact.

For each $x\in X$, let
\[
T_G^+\of{x}
=
\setOf{\mathbf{x}=\sequenceOf{x_n}:{n\in\naturals}\in X_G^+}:
{x_0=x}
\]
denote the collection of trajectories of $x$. If
$\pi_0\colon X^{\naturals}\to X$ denotes the zeroth-coordinate
projection, then the \emph{viability kernel} of $G$ is
\[
D_G
=
\operatorname{Viab}_{\domain{G}}\of{G}
=
\setOf{x\in\domain{G}}:{T_G^+\of{x}\neq\emptySet}
=
\pi_0\of{X_G^+}.
\]

\begin{observation}\label{observation:viability}
Let $\p{X,G}$ be a CR-dynamical system. Then:
\begin{enumerate}
\item $D_G$ is compact;
\item for each $x\in X$,
\[
x\in D_G
\quad\Longleftrightarrow\quad
G^n\of{x}\neq\emptySet
\text{ for every positive integer }n;
\]
\item every trajectory of $G$ lies entirely in $D_G$;
\item $D_{G^k}=D_G$ for every positive integer $k$.
\end{enumerate}
\end{observation}

Define the \emph{viable restriction} of $G$ by
\[
G_D
=
G\intersect\p{D_G\times D_G}.
\]
By Observation~\ref{observation:viability},
$\domain{G_D}=D_G$, and hence $G_D$ is the graph of the upper
semi-continuous set-valued function
\[
F_D\colon D_G\to2^{D_G},
\qquad
F_D\of{x}=G\of{x}\intersect D_G.
\]
Moreover,
\[
X_G^+=X_{G_D}^+.
\]
Thus, $G$ and its viable restriction have precisely the same infinite trajectories. They need not, however, have the same pseudo-orbits (see Section~\ref{section:ij-shadowing}),
since the definitions for $G$ use the original relation, including
branches outside $D_G$.

\subsection{Mahavier dynamical systems}
We equip $X^{\naturals}$ with the metric
\[
\rho\of{
\sequenceOf{x_n}:{n\in\naturals},
\sequenceOf{y_n}:{n\in\naturals}}
=
\sum_{n=0}^{\infty}
\Fraction{d\of{x_n,y_n}}/{2^{n+1}},
\]
and equip $X_G^+$ with the corresponding subspace metric.

Let
\[
\sigma\colon X^{\naturals}\to X^{\naturals}
\]
be the coordinate shift defined by
\[
\sigma\of{\sequence{x_0,x_1,x_2,\ldots}}
=
\sequence{x_1,x_2,x_3,\ldots}.
\]
Its restriction to $X_G^+$ is denoted by
\[
\sigma_G^+\colon X_G^+\to X_G^+.
\]
When $X_G^+\neq\emptySet$, we call
\[
\p{X_G^+,\sigma_G^+}
\]
the Mahavier dynamical system associated to $\p{X,G}$.

%----------------------------------------------------------------
\section{$\p{i, j}$-shadowing properties}\label{section:ij-shadowing}
We firstly go over definitions, generalising the shadowing property to CR-dynamical systems. These properties are invariant under topological conjugacy (see~\cite[Definition $9.3$]{minimality_CR}). We consider $\p{X, d}$ a metric space. 

\begin{definition}
Let $\p{X, G}$ be a CR-dynamical system and $\delta > 0$. We say a sequence $\sequenceOf{x_n}:{n \in \naturals}$ in $\domain{G}$ is a \emph{$\p{\delta, i}$-pseudo-orbit}, provided $d\of{x_{n+1}, y} \leq \delta$ for
\begin{itemize}
    \item[{\color{red} $\p{i = 1}$}] \textbf{all}  $y \in G\of{x_n}$,
    \item[{\color{red} $\p{i = 2}$}] \textbf{some}  $y \in G\of{x_n}$,
\end{itemize}
for all $n \in \naturals$. 
\end{definition}

\begin{definition}
Let $\p{X, G}$ be a CR-dynamical system and $\epsilon > 0$. We say $y \in D_G$ \emph{$\p{\epsilon, j}$-shadows} a sequence $\sequenceOf{x_n}:{n \in \naturals}$, provided for
\begin{itemize}
    \item[{\color{red} $\p{j = 1}$}] \textbf{all}  $\sequenceOf{y_n}:{n \in \naturals} \in T_G^+\of{y}$,
    \item[{\color{red} $\p{j = 2}$}] \textbf{some}  $\sequenceOf{y_n}:{n \in \naturals} \in T_G^+\of{y}$,
\end{itemize}
we have $d\of{x_n, y_n} < \epsilon$ for each $n \in \naturals$.
\end{definition}

\begin{definition}\label{def:metric-ij-shadowing}
A CR-dynamical system $\p{X, G}$ has the \emph{$\p{i, j}$-shadowing property}, if for every $\epsilon > 0$ there exists a $\delta > 0$ such that every $\p{\delta, i}$-pseudo-orbit is $\p{\epsilon, j}$-shadowed by a point in $X$. 
\end{definition}

The $\p{1, j}$-shadowing properties require some care. A $\p{\delta,1}$-pseudo-orbit can pass through a point $x$ only if $G\of{x}$ is contained in a closed ball of radius $\delta$, whose centre is again in $\domain{G}$. Consequently, there may be no $\p{\delta,1}$-pseudo-orbits at sufficiently small scales, in which case the corresponding shadowing properties hold vacuously. We now give a criterion which detects this phenomenon.

For each $\delta>0$, define a relation $Q_\delta$ on $\domain{G}$ by
\[
Q_\delta
=
\setOf{\p{x,z}\in\domain{G}\times\domain{G}}:
{G\of{x}\subseteq\overline{B}\of{z,\delta}},
\]
and, for each $x\in\domain{G}$, define
\[
r_G\of{x}
=
\inf_{z\in\domain{G}}
\sup_{y\in G\of{x}}d\of{y,z}.
\]
Thus, $r_G\of{x}$ is the infimal radius of a closed ball, centred at a
point of $\domain{G}$, which contains $G\of{x}$.

\begin{observation}
The $\p{\delta,1}$-pseudo-orbits of $\p{X,G}$ are precisely the
infinite paths of $Q_\delta$. Consequently:
\begin{enumerate}
\item if there exists $\delta>0$ such that $Q_\delta$ has no infinite
path, then the $\p{1,1}$- and $\p{1,2}$-shadowing properties both hold
vacuously;
\item if
\[
\inf_{x\in\domain{G}}r_G\of{x}>0,
\]
then, for every sufficiently small $\delta>0$, there are no
$\p{\delta,1}$-pseudo-orbits;
\item every $\p{\delta,1}$-pseudo-orbit
$\sequenceOf{x_n}:{n\in\naturals}$ satisfies
\[
\emph{diam}\of{G\of{x_n}}\leq2\delta
\]
for every $n\in\naturals$;
\item if there exists a sequence
$\sequenceOf{x_n}:{n\in\naturals}$ such that
\[
G\of{x_n}=\set{x_{n+1}}
\]
for every $n\in\naturals$, then $\p{\delta,1}$-pseudo-orbits exist
for every $\delta>0$.
\end{enumerate}
\end{observation}

We conclude this section with several more observations.

\begin{observation}\label{observation:implication}
The $\p{i, j}$-shadowing properties have the following implication diagram. 
\[
 \begin{tikzpicture}[
       decoration = {markings,
                     mark=at position .5 with {\arrow{Stealth[length=2mm]}}},
       dot/.style = {circle, fill, inner sep=2.4pt, node contents={},
                     label=#1},
every edge/.style = {draw, postaction=decorate}
                        ]
\node (21) {$\p{2, 1}$};
\node (22) at (1.5,-1.5) {$\p{2, 2}$};
\node (11) at (-1.5,-1.5) {$\p{1, 1}$};
\node (12) at (0, -3) {$\p{1, 2}$};

\draw[->, thick]   (21) -- (22);
\draw[->, thick]   (21) -- (11);
\draw[->, thick]   (11) -- (12);
\draw[->, thick]   (22) -- (12);
 \end{tikzpicture}
\]
\end{observation}

\begin{observation}\label{observation:extension}
Suppose $\p{X, G}$ and $\p{X, H}$ are CR-dynamical systems such that $G \subseteq H$ and  $\domain{G} = \domain{H}$. If $\p{X, G}$ has the $\p{1, 2}$-shadowing property, then $\p{X, H}$ has the $\p{1, 2}$-shadowing property. 
\end{observation}

\begin{observation}\label{observation:contraction}
Suppose $\p{X, G}$ and $\p{X, H}$ are CR-dynamical systems such that $G \subseteq H$ and $D_G = D_H$. If $\p{X, H}$ has the $\p{2, 1}$-shadowing property, then $\p{X, G}$ has the $\p{2, 1}$-shadowing property.
\end{observation}

%----------------------------------------------------------------
\section{Closed relations containing the diagonal}\label{section:diagonal}

In this section we prove Theorems~\ref{theorem:main-A} and~\ref{theorem:main-B}. We begin by proving several propositions, which together yield Theorem~\ref{theorem:main-A}.  

\begin{proposition}\label{proposition:diagonal-totally-disconnected}
Suppose $\p{X, G}$ is a CR-dynamical system with $\Delta_X \subseteq G$. Then, 
\begin{enumerate}
    \item if $\p{X, G}$ has the $\p{2, 1}$-shadowing property, then $X$ is totally disconnected; 
     \item if $X$ is totally disconnected, then $\p{X, G}$ has the $\p{1, 2}$-shadowing property.
\end{enumerate}
\end{proposition}
\begin{proof}
Suppose $\p{X, G}$ has the $\p{2, 1}$-shadowing property.  By Observation~\ref{observation:contraction}, $\p{X, \Delta_X}$ has the $\p{2, 1}$-shadowing property, implying $\p{X, \text{id}_X}$ has the shadowing property (where $\text{id}_X$ is the identity map on $X$).  Thus, $X$ is totally disconnected by~\cite[Theorem~3]{meddaugh_shadowing_recurrence}, establishing $\p{1}$. 

Suppose $X$ is totally disconnected.  By~\cite[Theorem~3]{meddaugh_shadowing_recurrence}, $\p{X, \text{id}_X}$ has the shadowing property, implying $\p{X, \Delta_X}$ has the $\p{1, 2}$-shadowing property.  By Observation~\ref{observation:extension}, it follows $\p{X, G}$ has the $\p{1, 2}$-shadowing property, yielding $\p{2}$. 
\end{proof}

\begin{proposition}\label{prop:diagonal-plus-horizontal}
Let $\p{X, G}$ be a CR-dynamical system such that $G = \Delta_X \union \p{X \times \set{x}}$ for some $x \in X$. Then $\p{X, G}$ has the $\p{1, 1}$-shadowing property. 
\end{proposition}
\begin{proof}
Suppose $\epsilon > 0$.  Let $\delta \in \p{0, \Fraction \epsilon / 2}$.  If $\sequenceOf{x_n}:{n \in \naturals}$ is a $\p{\delta, 1}$-pseudo-orbit, then $d\of{x, x_{n + 1}} \leq \delta$ and $d\of{x_n, x_{n+1}} \leq \delta$ for every $n \in \naturals$.  Hence, 
\[
d\of{x, x_n} \leq d\of{x, x_{n+1}} + d\of{x_{n+1}, x_n} \leq 2\delta < \epsilon
\]
for every $n \in \naturals$, implying $x$ $\p{\epsilon, 1}$-shadows $\sequenceOf{x_n}:{n \in \naturals}$.  Thus, $\p{X, G}$ has the $\p{1, 1}$-shadowing property. 
\end{proof}

\begin{proposition}\label{prop:id-shadowing-diagonal-totally_disconnected-counterexamples}
Suppose $\p{X, G}$ is a CR-dynamical system with $\Delta_X \subseteq G$ and $j \in \set{1, 2}$. Then, 
\begin{enumerate}
    \item if $\p{X, G}$ has the $\p{2, 2}$-shadowing property, then $X$ need not be totally disconnected; 
    \item if $\p{X, G}$ has the $\p{1, j}$-shadowing property, then $X$ need not be totally disconnected;
    \item if $X$ is totally disconnected, then $\p{X, G}$ need not have the $\p{2, 1}$-shadowing property; 
    \item if $X$ is totally disconnected, then $\p{X, G}$ need not have the $\p{1, 1}$-shadowing property;
    \item if $X$ is totally disconnected, then $\p{X, G}$ need not have the $\p{2, 2}$-shadowing property.  
\end{enumerate}
\end{proposition}
\begin{proof}
For $\p{1}$, take $X = [0, 1]$ and $G = X \times X$.   

For $\p{2}$, take $X = [0, 1]$ and $G = \Delta_X \union \p{X \times \set{0}}$, then apply Proposition~\ref{prop:diagonal-plus-horizontal}. 

For $\p{3}$, take $X = \set{0, 1}$ with discrete topology and $G = X \times X$.  

For $\p{4}$, take $X$ to be the Cantor set and $G = \Delta_X \union \p{X \times \set{0}} \union \p{\set{0} \times X}$. Choose $\epsilon = \Fraction 1 / 2$, and let $\delta > 0$ be arbitrary. There exists $n \in \naturals$ such that $\Fraction 4 / {3^n} \leq \delta$. Observe $\sequenceOf{\Fraction 2 / {3^k}}:{k \geq n}$ is a $\p{\delta, 1}$-pseudo-orbit. To derive a contradiction, suppose there exists $y \in X$ which $\p{\epsilon, 1}$-shadows $\sequenceOf{\Fraction 2 / {3^k}}:{k \geq n}$. Notice $0, 1 \in G^2\of{y}$. But this means $d\of{\Fraction 2 / {3^{n+2}}, 0} < \epsilon$ and $d\of{\Fraction 2 / {3^{n+2}}, 1} < \epsilon$. Applying the triangle inequality yields $d\of{0, 1} < 1$, a contradiction. Thus, $\p{X, G}$ does not have the $\p{1, 1}$-shadowing property. 

For $\p{5}$, take $X$ to be the Cantor set, and $G = \Delta_X \union \setOf{\p{a_n, a_{n+1}}}:{n \in \naturals} \union \setOf{\p{b_n, b_{n+1}}}:{n \in \integers}$, where $\sequenceOf{a_n}:{n \in \naturals}$ and $\sequenceOf{b_n}:{n \in \integers}$ are sequences in $X$ with the properties,
\begin{itemize}
    \item for each $\p{n, m} \in \naturals \times \integers$, $a_n < a_{n+1}$ and $b_m < b_{m+1}$;
    \item $a_0 = 0$, $\lim_{n \to +\infinity} a_n = \Fraction 2 / 9 = \lim_{n\to -\infinity} b_n$ and $\lim_{n \to +\infinity} b_n = \Fraction 1 / 3$. 
\end{itemize}
Indeed, $\p{X, G}$ is a CR-dynamical system. Now, let $\epsilon = \Fraction 1 / {27}$, and $\delta > 0$ be arbitrary. There exists $\p{n, m} \in \naturals \times \integers$ such that $\sequence{a_0, \ldots, a_n, b_m, b_{m+1}, \ldots}$ is a $\p{\delta, 2}$-pseudo-orbit. However, it is not $\p{\epsilon, 2}$-shadowed by any point in $X$: to derive a contradiction, suppose $y$ $\p{\epsilon, 2}$-shadows $\sequence{a_0, \ldots, a_n, b_m, b_{m+1}, \ldots}$. Then, $d\of{0, y} < \Fraction 1/ {27}$. This implies $G^k\of{y} \subseteq [0, \Fraction 2 / 9] \intersect X$ for all $k \in \naturals$. Since $\lim_{k \to +\infinity} b_k = \Fraction 1 / 3$, there exists $\ell \geq m$ such that $d\of{\Fraction 1 /3, b_\ell} < \Fraction 1/ {27}$. There exists $y_\ell \in G^k\of{y}$ for some $k \in \naturals$, such that $d\of{b_\ell, y_\ell} < \Fraction 1 / {27}$, by the assumption $y$ $\p{\epsilon, 2}$-shadows our sequence. But then
\[
\Fraction 1 / 3 = d\of{\Fraction 1 / 3, 0} \leq d\of{\Fraction 1 / 3, b_{\ell}} + d\of{b_\ell, y_\ell} + d\of{y_\ell, 0} < \Fraction 1 / {27} + \Fraction 1 / {27} + \Fraction 2 / {9} = \Fraction 8 / {27}, 
\]
a contradiction. 
\end{proof}

%We now make an observation on $\p{1}$ in Theorem~\ref{thm:diagonal-totally-disconnected}, which generalises the forward direction in Theorem~\ref{thm:id-shadowing-iff-totally-disconnected}, showing the condition implies $G$ is the graph of a continuous function. Indeed, the same holds when generalising the forward direction of Theorem~\ref{thm:periodic-shadowing-iff-totally-disconnected}. 

\begin{proposition}\label{prop:21-implies-diagonal}
Suppose $\p{X, G}$ is a CR-dynamical system with $\Delta_X \subseteq G$. If $\p{X, G}$ has the $\p{2, 1}$-shadowing property, then $G = \Delta_X$. 
\end{proposition}
\begin{proof}
Suppose $\p{X, G}$ has the $\p{2, 1}$-shadowing property. To derive a contradiction, suppose $G \neq \Delta_X$.  Then, there exists $x \in X$ such that there are $y \neq z \in G\of{x}$. There exists $\epsilon > 0$ such that $d\of{y, z} \geq \epsilon$. Let $\eta \in \p{0, \Fraction \epsilon / 4}$.  There exists $\delta > 0$ satisfying the $\p{2, 1}$-shadowing property for $\eta$. Observe $\sequence{x, y, y, \ldots}$ and $\sequence{x, z, z, \ldots}$ are both $\p{\delta, 2}$-pseudo-orbits. There exists $u, v \in X$ which $\p{\eta, 1}$-shadow $\sequence{x, y, y, \ldots}$ and $\sequence{x, z, z \ldots}$, respectively.  Then, 
\[
d\of{y, z} \leq d\of{y, u} + d\of{u, x} + d\of{x, v} + d\of{v, z} < 4\eta < \epsilon,
\]
a contradiction.  Thus, it must be the case $G = \Delta_X$, as desired.  
\end{proof}

Indeed, together Propositions~\ref{proposition:diagonal-totally-disconnected}, ~\ref{prop:id-shadowing-diagonal-totally_disconnected-counterexamples}, and~\ref{prop:21-implies-diagonal} give Theorem~\ref{theorem:main-A}. 

We now work towards proving Theorem~\ref{theorem:main-B}, which generalises~\cite[Theorem~5]{meddaugh_shadowing_recurrence}.  We begin with generalising the following well-known result to CR-dynamical systems, which we require for the proof. 

\begin{lemma}\label{lem:invariance-under-comp-topological}
Let $\p{X, f}$ be a topological dynamical system. Then $\p{X, f}$ has the shadowing property if, and only if, $\p{X, f^k}$ has the shadowing property for all $k \geq 1$. 
\end{lemma}

Note Lemma~\ref{lem:invariance-under-comp-topological} has been generalised in set-valued dynamical systems~\cite[Proposition $3.2$]{Luo2020}. 

If $D_G \neq \emptySet$, then $\p{X, G^k}$ is a CR-dynamical system for each $k \geq 1$. We now generalise Lemma~\ref{lem:invariance-under-comp-topological} below.

\begin{lemma}\label{lem:invariance-under-comp}
Let $\p{X, G}$ be a CR-dynamical system such that $D_G \neq \emptySet$, and $j \in \set{1, 2}$. Then, $\p{X, G}$ has the $\p{2, j}$-shadowing property if, and only if, $\p{X, G^k}$ has the $\p{2, j}$-shadowing property for all positive integers $k$.
\end{lemma}
\begin{proof}
Necessity is obvious, so we only provide proof of the forward direction; to this end, suppose $\p{X, G}$ has the $\p{2, j}$-shadowing property,  and that $k$ is a positive integer. Let $\epsilon > 0$. Then, there exists $\delta > 0$ satisfying the $\p{2, j}$-shadowing property.  Suppose $\sequenceOf{x_n}:{n \in \naturals}$ is a $\p{\delta, 2}$-pseudo-orbit in $\p{X, G^k}$. That is to say, for each $n \in \naturals$, there exists $y \in G^k\of{x_n}$ such that $d\of{y, x_{n+1}} \leq \delta$. Construct a sequence $\sequenceOf{y_n}:{n \in \naturals}$ in $\domain{G}$ such that we take $y_0 = x_0$, $y_1 \in G\of{y_0}$, $y_2 \in G\of{y_1}$, $\ldots$, $y_{k-1} \in G\of{y_{k-2}}$, $y_k = x_1$, and so on (observe $y_{nk} = x_n$ for each $n \in \naturals$). Choose this in such a way so that, for each positive integer $n$, $y_{nk - 1} \in G^{k-1}\of{x_{n-1}}$ such that there is $y \in G\of{y_{nk-1}}$ with $d\of{y, x_n} \leq \delta$ (which we can do because $\sequenceOf{x_n}:{n \in \naturals}$ is a $\p{\delta, 2}$-pseudo-orbit of $\p{X, G^k}$). 

Now, we claim that, for each $n \in \naturals$, there exists $y \in G\of{y_n}$ such that $d\of{y, y_{n+1}} \leq \delta$. That is to say, we claim that $\sequenceOf{y_n}:{n \in \naturals}$ is a $\p{\delta, 2}$-pseudo-orbit in $\p{X, G}$. Notice that if $n + 1$ is not divisible by $k$, that $y_{n+1} \in G\of{y_n}$. If on the other hand $n + 1$ is divisible by $k$, say $n + 1 = mk$, then $y_{n+1} = x_m$ and $y_n \in G^{k-1}\of{x_{m-1}}$. By our construction, there exists $y \in G\of{y_{n}} \subseteq G^k\of{x_{m-1}}$ with $d\of{y, y_{n+1}} = d\of{y, x_m} \leq \delta$. Hence, the claim follows. 

Thus, there exists $z \in D_G = D_{G^k}$ which $\p{\epsilon, j}$-shadows $\sequenceOf{y_n}:{n \in \naturals}$ in $\p{X, G}$. Furthermore, it is clear $z$ $\p{\epsilon, j}$-shadows $\sequenceOf{x_n}:{n \in \naturals}$ in $\p{X, G^k}$, and thus $\p{X, G^k}$ has the $\p{2, j}$-shadowing property. 
\end{proof}

\begin{example}\label{ex:comp-fails}
We show Lemma~\ref{lem:invariance-under-comp} fails for the $\p{1, j}$-shadowing properties. Let $\p{X, G}$ be a CR-dynamical system for which $X = [0, 1] \union \set{2}$ and 
\[
G = \p{\left[0, \Fraction 1 / 2\right] \times \set{1}} \union L_1 \union L_2 \union \set{\p{2, 2}}, 
\]
where 
\[
L_1 = \setOf{\p{x, y} \in \left[0, \Fraction 1 / 2\right] \times X}:{y = x + \Fraction 1 / 2} \text{ and } L_2 = \setOf{\p{x, y} \in \left[\Fraction 1 / 2, \Fraction 3 / 4\right] \times X}:{y = x - \Fraction 1 / 2}.
\] 
Then, 
\[
G^2 = \set{\p{0, 1}} \union \left[\Fraction 1 / 2, \Fraction 3 / 4\right] \times \set{1} \union \Delta_X^\prime, 
\]
where 
\[
\Delta_X^\prime = \Delta_X \intersect \p{\left[0, \Fraction 1 / 4\right]^2 \union \left[\Fraction 1 / 2, \Fraction 3 / 4\right]^2 \union \set{\p{2, 2}}}. 
\]
Both $G$ and $G^2$ are shown in Figure~\ref{fig:comp-fails}. 

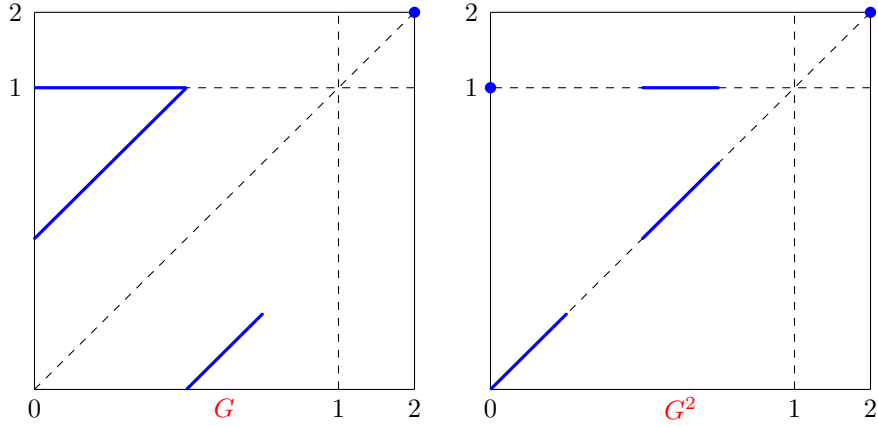
\begin{figure}[h]
\[
 \begin{tikzpicture}[
       decoration = {markings,
                     mark=at position .5 with {\arrow{Stealth[length=2mm]}}},
       dot/.style = {circle, fill, inner sep=0pt, node contents={},
                     label=#1},
       dots/.style = {circle, fill, inner sep=1.5pt, node contents={},
     label=#1},
every edge/.style = {draw, postaction=decorate}
                        ]
%\node (a) [dot];
%\node (b) at (4, 0) [dot];
%\node (c) at (4, 4) [dot]; 
%\node[blue] (d) at (0, 4) [dot];
%\node (x) at (2, 0) [dot];
%\node (y) at (2, 4) [dot];
\node[blue] (z) at (0, 2) [dot];
\node[blue] (k) at (3, 1) [dot];
\node (e) at (2.5, -.25) {$\color{red}\text{$G$}$};

\node (a0) at (0, -.25) {$0$};
\node (a1) at (4, -.25) {$1$};
\node (a2) at (-.25, 4) {$1$};
\node (a3) at (-.25, 5) {$2$};
\node (a4) at (5, -.25) {$2$};

%\node (a1) at (6, 0) [dot];
%\node (b1) at (10, 0) [dot];
%\node (c1) at (10, 4) [dot]; 
%\node (d1) at (6, 4) [dot];
%\node (x1) at (8, 0) [dot];
%\node (y1) at (8, 4) [dot];
\node[blue] (z1) at (8, 2) [dot];
\node[blue] (z2) at (9, 3) [dot];
\node[blue] (z3) at (8, 4) [dot];
\node[blue] (z4) at (9, 4) [dot];
\node[blue] (z5) at (7, 1) [dot];
\node (e1) at (8.5, -.25) {$\color{red}\text{$G^2$}$};

\node (b0) at (6, -.25) {$0$};
\node (b1) at (10, -.25) {$1$};
\node (b2) at (5.75, 4) {$1$};
\node (b3) at (5.75, 5) {$2$};
\node (b4) at (11, -.25) {$2$};

\draw (0, 0) -- (4, 0);
\draw[dashed] (4, 0) -- (4, 5);
\draw[dashed] (5, 4) -- (0, 4);
\draw (0, 4) -- (0, 0);
\draw[dashed] (0, 0) -- (5, 5);
\draw (4, 0) -- (5, 0);
\draw (0, 4) -- (0, 5); 
\draw (5, 0) -- (5, 5);
\draw (0, 5) -- (5, 5);
\draw[blue, very thick] (0, 4) -- (2, 4);
\draw[blue, very thick] (0, 2) -- (2, 4);
\draw[blue, very thick] (2, 0) -- (3, 1);

\draw (6, 0) -- (10, 0);
\draw (10, 0) -- (11, 0);
\draw (6, 4) -- (6, 5);
\draw (6, 5) -- (11, 5);
\draw (11, 0) -- (11, 5);
\draw[dashed] (6, 0) -- (11, 5);
\draw[dashed] (10, 0) -- (10, 5);
\draw[dashed] (11, 4) -- (6, 4);
\draw (6, 4) -- (6, 0);
\draw[blue, very thick] (8, 2) -- (9, 3);
\draw[blue, very thick] (8, 4) -- (9, 4);
\draw[blue, very thick] (6, 0) -- (7, 1);

\node[blue] (z6) at (6, 4) [dots];

\node[blue] (z7) at (5, 5) [dots];

\node[blue] (z8) at (11, 5) [dots];
 \end{tikzpicture}
\]
\caption{The relations $G$ and $G^2$ from Example~\ref{ex:comp-fails}}\label{fig:comp-fails}
\end{figure}

To see $\p{X, G}$ has the $\p{1, j}$-shadowing property, we show $\p{X, G}$ has exactly one $\p{\delta, 1}$-pseudo-orbit for $\delta \in \p{0, \Fraction 1/ 4}$.  In particular, we claim $\sequence{2, 2, \ldots}$ is the only $\p{\delta, 1}$-pseudo-orbit. Note firstly it is clear $\sequence{2, 2, \ldots}$ is a $\p{\delta, 1}$-pseudo-orbit, and that it is $\p{\epsilon, j}$-shadowed by $2$, for each $\epsilon > 0$.  Now, to this end, we suppose for contradiction that $\sequenceOf{x_n}:{n \in \naturals}$ is a different $\p{\delta, 1}$-pseudo-orbit of $\p{X, G}$. Then, for each $n$, $x_n \in [0, \Fraction 1/ 2]$ or $x_n \in (\Fraction 1 / 2, \Fraction 3 / 4]$. In the former case, $1 \in G\of{x_n}$, implying $d\of{1, x_{n+1}} \leq \delta$. However, $\domain{G} = [0, \Fraction 3/ 4] \union \set{2}$, and $\delta \in \p{0, \Fraction 1 / 4}$, so this cannot be the case. So, we must have $x_n \in (\Fraction 1 / 2, \Fraction 3 / 4]$ for each $n \in \naturals$, implying $d\of{x_n - \Fraction 1 / 2, x_{n+1}} \leq \delta$ for each $n \in \naturals$. However, this implies $x_{n+1} \in [0, \Fraction 1 / 2]$, so we reach the former case, a contradiction.  It follows $\p{X, G}$ has the $\p{1, j}$-shadowing property. 

To see why $\p{X, G^2}$ does not have $\p{1, j}$-shadowing property, we take $\epsilon = \Fraction 1 / {16}$ and $\delta > 0$, and consider $\p{\delta, 1}$-pseudo-orbit $\sequence{\Fraction 1 / {4n}, \Fraction 2 / {4n}, \ldots, \Fraction 1 / 4, \Fraction 1 / 4, \ldots}$, where $\Fraction 1 / n \leq \delta$. Then, if $x$ $\p{\epsilon, j}$-shadows the sequence, we have $d\of{x, \Fraction 1/ {4n}} < \Fraction 1 / {16}$ and $d\of{x, \Fraction 1 / 4} < \Fraction 1 / {16}$. But by the triangle inequality we obtain $\Fraction 1 / 4 - \Fraction 1 / {4n} < \Fraction 1 / 8$, a contradiction (choosing $n$ large enough). Thus, $\p{X, G^2}$ does not have the $\p{1, j}$-shadowing property. 
\end{example}

We now work towards proving $\p{1}$ of Theorem~\ref{theorem:main-B}. It is not necessary to consider the case when $G^n = \Delta_X$ for some positive integer $n$, for this implies $G$ is the graph of a continuous self-map on $X$. 

\begin{proposition}\label{prop:G^n=diagonal-implies-G-single-valued}
Let $\p{X, G}$ be a CR-dynamical system such that $G^n = \Delta_X$ for some positive integer $n$. Then, $G$ is the graph $\Graph{f}$ for a periodic homeomorphism $f \colon X \to X$. 
\end{proposition}
\begin{proof}
We claim each point in $X$ has exactly one trajectory in $\p{X, G}$. Clearly, $D_G = X$, since $G^n = \Delta_X$. To derive a contradiction, suppose there exists $x \in X$ with more than one trajectory. Now, $G^n\of{x} = \set{x}$. Since $x$ has more than one trajectory, $n > 1$. In particular, there must exist $y \neq z \in G^k\of{x}$ for some $k \in \set{0, \ldots, n - 1}$. Notice $x \in G^{n - k}\of{y}$, implying $z \in G^n\of{y}$. But $G^n\of{y} = \set{y}$, 
 implying $y = z$, a contradiction.  Hence, our claim follows. 

By our claim and the fact $D_G = X$, it follows $G\of{x}$ is a singleton for each $x \in X$. Define $f \colon X \to X$ by $f\of{x} = y$, where $G\of{x} = \set{y}$, for each $x \in X$. Observe $\Graph{f} = G$, and that $f$ is a continuous bijection (and hence a homeomorphism). 
\end{proof}

\begin{proposition}\label{proposition:21-implies-function}
Suppose $\p{X, G}$ is a CR-dynamical system with $\Delta_X \subseteq G^n$ for some positive integer $n$. If $\p{X, G}$ has the $\p{2, 1}$-shadowing property, then $G = \Graph{f}$ for a periodic homeomorphism $f \colon X \to X$ with $f^n = \emph{id}_X$, and $X$ is totally disconnected.
\end{proposition}
\begin{proof}
Suppose $\p{X, G}$ has the $\p{2, 1}$-shadowing property. By Lemma~\ref{lem:invariance-under-comp}, $\p{X, G^n}$ has the $\p{2, 1}$-shadowing property. By Propositions~\ref{prop:21-implies-diagonal} and~\ref{prop:G^n=diagonal-implies-G-single-valued}, $G=\Graph{f}$ for a homeomorphism satisfying
$f^n=\text{id}_X$. Since the $\p{2,1}$-shadowing property agrees with classical shadowing for single-valued systems, $f$ has the shadowing property. Hence $X$ is totally disconnected by~\cite[Theorem~5]{meddaugh_shadowing_recurrence}.
\end{proof}

We now prove $\p{2}$ of Theorem~\ref{theorem:main-B}. 

\begin{proposition}
Suppose $\p{X, G}$ is a CR-dynamical system with $\Delta_X \subseteq G^n$ for some positive integer $n$. If $X$ is totally disconnected and $F_k$ is uniformly upper semi-continuous for each $k \in \set{0, \ldots, n}$, then $\p{X, G}$ has the $\p{1, 2}$-shadowing property.
\end{proposition}
\begin{proof}
We abstract the proof strategies in~\cite[Theorem~4.1]
{topological_chain_uniform} and~\cite[Theorem~5]
{meddaugh_shadowing_recurrence}. Suppose $X$ is totally
disconnected and that $F_k$ is uniformly upper semi-continuous for
each $k\in\set{0,\ldots,n}$.

We show that $\p{X,\Graph{F_1}}$ has the
$\p{1,2}$-shadowing property. Let $\epsilon>0$. Since $X$ is a
compact totally disconnected metric space, there exists a finite
pairwise disjoint clopen cover $\collection{U}$ of $X$ such that
every member of $\collection{U}$ has diameter strictly less than
$\epsilon$. Choose $\eta > 0$ such that
\[
d\of{x, y} < \eta
\ \ \Longrightarrow \ \
\text{$x$ and $y$ belong to the same member of $\collection{U}$.}
\]
Such an $\eta$ exists because $\collection{U}$ is finite and its
members are pairwise disjoint compact subsets of $X$.

Since $F_k$ is uniformly upper semi-continuous for each
$k \in \set{0, \ldots, n}$, there exists $\delta \in \p{0, \eta}$ such that
\[
d\of{x, y} \leq \delta
\ \ \Longrightarrow \ \
F_k\of{x}\subseteq
\Union_{z\in F_k\of{y}} B\of{z,\eta}
\]
for each $x,y \in X$ and every $k \in \set{0, \ldots, n}$. Since
$F_0 = F_n = \text{id}_X$, extend the notation by setting $F_k := F_{k \bmod n}$ for each $k \in \naturals$. 
For a subset $A \subseteq X$, write
\[
F_k(A)=\Union_{a\in A} F_k\of{a}. 
\]

Suppose $\sequenceOf{x_i}:{i\in\naturals}$ is a
$\p{\delta,1}$-pseudo-orbit of $\p{X,\Graph{F_1}}$. Then
\[
d\of{a,x_{i+1}} \leq \delta
\]
for every $i \in \naturals$ and every $a \in F_1\of{x_i}$.
Consequently, uniform upper semi-continuity gives
\[
F_k\of{F_1\of{x_i}}
\subseteq
\Union_{z \in F_k\of{x_{i+1}}} B\of{z,\eta}
\]
for every $i,k\in\naturals$.

We next show that
\[
F_{k+1}\of{x} \subseteq F_k\of{F_1\of{x}}
\]
for every $x\in X$ and $k\in\naturals$. We need only consider $k \in \set{0, \ldots, n - 1}$. Let
$y\in F_{k+1}\of{x}$. There exists 
\[
\sequence{x=w_0,w_1,\ldots,w_{k+1}=y,\ldots,w_n=x} \in \mahavier{n}{G}.
\]
Notice $z=w_1 \in F_1\of{x}$ and $y \in F_k\of{z}$. Consequently,
\[
y\in F_k\of{z} \subseteq F_k\of{F_1\of{x}},
\]
as required. It follows that
\[
F_{k+1}\of{x_i}
\subseteq
\Union_{z\in F_k\of{x_{i+1}}} B\of{z,\eta}
\]
for every $i,k\in\naturals$.

We claim that, for every $i \in \naturals$, the set $F_i\of{x_0}$ and
the point $x_i$ lie in the same member of $\collection{U}$. This is
immediate when $i = 0$. For $i \geq 1$, we have
\[
F_1\of{x_{i-1}}
\subseteq
\Union_{z\in F_0\of{x_i}} B\of{z,\eta}
=
B(x_i,\eta).
\]
Hence, $F_1\of{x_{i-1}}$ and $x_i$ lie in the same member of
$\collection{U}$ by choice of $\eta$. Similarly,
\[
F_2\of{x_{i-2}}
\subseteq
\Union_{z\in F_1\of{x_{i-1}}} B\of{z,\eta},
\]
so $F_2\of{x_{i-2}}$ lies in that same member of $\collection{U}$.
Continuing this way inductively proves our claim. 

Since $\Delta_X \subseteq G^n$, choose 
\[
\sequence{w_0,w_1,\ldots,w_n} \in \mahavier{n}{G}
\]
with $w_0=w_n=x_0$. Define
\[
z_i=w_{i\bmod n}
\]
for every $i\in\naturals$. Observe $z_{i+1} \in F_1\of{z_i}$ for every $i \in \naturals$. Thus,
\[
\sequenceOf{z_i}:{i\in\naturals}
\in T_{\Graph{F_1}}^+\of{x_0}.
\]
Furthermore, $z_i \in F_i\of{x_0}$ for every $i \in \naturals$. By the
claim, $z_i$ and $x_i$ belong to the same member of $\collection{U}$,
and hence
\[
d\of{z_i,x_i} < \epsilon
\]
for every $i \in \naturals$. Therefore,
$z_0$ $\p{\epsilon, 2}$-shadows
$\sequenceOf{x_i}:{i\in\naturals}$. It follows that $\p{X,\Graph{F_1}}$ has the
$\p{1,2}$-shadowing property. Since
$\Graph{F_1}\subseteq G$, Observation~\ref{observation:extension}
implies that $\p{X,G}$ has the $\p{1,2}$-shadowing property.
\end{proof}

\begin{note}
There are examples where it is not the case $F_k$ is uniformly upper semi-continuous for each $k \in \set{0, \ldots, n}$, yet $\p{X, G}$ has the $\p{1, 2}$-shadowing property. For example, take $X$ to be the Cantor set and $G = \p{X \times \set{0}} \union \p{\set{0} \times X}$. By Observation~\ref{observation:extension}, $\p{X, G}$ has the $\p{1, 2}$-shadowing property, since the constant map sending everything to $0$ has shadowing. Moreover, $G^2 = X \times X$, and so clearly $\Delta_X \subseteq G^2$. Observe $\Graph{F_1} = G$. Indeed, $F_1$ is not uniformly upper semi-continuous (if $\delta > 0$, there exists $n \in \naturals$ such that $d\of{0, \Fraction 2/ {3^n}} \leq \delta$, but we may choose $\epsilon > 0$ such that $F_1\of{0} = X  \not\subseteq B\of{0, \epsilon}$). Due to this, and with no counter-examples known, we formulate the following question.  
\end{note}

\question{Suppose $\p{X, G}$ is a CR-dynamical system with $\Delta_X \subseteq G^n$ for some positive integer $n$. If $X$ is totally disconnected, must $\p{X, G}$ have the $\p{1, 2}$-shadowing property?}

%----------------------------------------------------------------
\section{Isometries and Mahavier dynamical systems}\label{section:isometries}
In this section we prove Theorems~\ref{theorem:main-C} and~\ref{theorem:main-D}. We begin by proving $\p{1}$ of  Theorem~\ref{theorem:main-C}. 

\begin{proposition}\label{proposition:isometry-mahavier}
Suppose $\p{X, f}$ is a topological dynamical system, where $f \colon X \to X$ is an isometry.  Then, the following are equivalent. 
\begin{enumerate}
\item $\p{X, f}$ has the shadowing property. 
\item $\p{X_\Graph{f}^+, \sigma_\Graph{f}^+}$ has the shadowing property.
\item $X$ is totally disconnected.
\item $X_\Graph{f}^+$ is totally disconnected. 
\end{enumerate}
\end{proposition}
\begin{proof}
We firstly show the equivalence between $\p{1}$ and $\p{3}$. The proof is similar to~\cite[Theorem $3$]{meddaugh_shadowing_recurrence}. To this end, firstly suppose $\p{X, f}$ has the shadowing property. To derive a contradiction, suppose $X$ has a non-trivial connected component $C$. Let $\epsilon > 0$ be less than the diameter of $C$, and let $x \neq y \in C$ be such that $d\of{x, y} > \epsilon$. There exists $\delta > 0$ satisfying the shadowing property for $\p{X, f}$, such that every $\delta$-pseudo-orbit is $\Fraction \epsilon / 2$-shadowed by a point. By~\cite[Exercise $4.23$ Part $\p{d}$]{nadler}, since $C$ is connected, there is a sequence $\sequenceOf{x_n}:{n \in \naturals}$ and $m > 0$ such that $x_0 = x$, $x_k = y$ for each $k \geq m$, and $d\of{x_n, x_{n+1}} < \delta$ for each $n \in \naturals$. Notice $\sequenceOf{f^n\of{x_n}}:{n \in \naturals}$ is a $\delta$-pseudo-orbit, since 
\[
d\of{f\of{f^n\of{x_n}}, f^{n+1}\of{x_{n+1}}} = d\of{x_n, x_{n+1}} \leq \delta
\]
for each $n \in \naturals$. There exists $z \in X$ which $\Fraction \epsilon / 2$-shadows our $\delta$-pseudo-orbit. However,  
\[
d\of{x, y} \leq d\of{x, z} + d\of{y, z} = d\of{x, z} + d\of{f^m\of{y}, f^m\of{z}} < \epsilon,
\]
a contradiction. Hence, $X$ is totally disconnected, establishing the implication from $\p{1}$ to $\p{3}$. 

Conversely, suppose $X$ is totally disconnected. Fix $\epsilon > 0$. As $X$ is a totally disconnected compact metric space, there is a finite cover $\collection{U}$ which consists of pairwise disjoint clopen sets, each with diameter strictly less than our chosen $\epsilon$. Now, set $\eta := \min\setOf{d\of{U, V}}:{U \neq V \in \collection{U}}$, and let $\delta \in \p{0, \eta}$ be arbitrary. Suppose $\sequenceOf{x_n}:{n \in \naturals}$ is a $\delta$-pseudo-orbit. By choice of $\delta$, if $f\of{x_n} \in U$ for some $U \in \collection{U}$, $d\of{f\of{x_n}, x_{n+1}} \leq \delta$ implies $d\of{x_{n+1}, U} \leq \delta$ and therefore $x_{n+1} \in U$. So, if $f\of{x_0} \in U$, $x_1 \in U$. But also, $d\of{f^2\of{x_0}, f\of{x_1}} \leq \delta$ and $d\of{f\of{x_1}, x_2} \leq \delta$, so if $f^2\of{x_0} \in U$ some $U \in \collection{U}$, then $f\of{x_1}$ and consequently $x_2$ are both in $U$. Continuing in this fashion, it follows $f^n\of{x_0}$ and $x_n$ lie in the same $U_n \in \collection{U}$ for each $n \in \naturals$. Since $d\of{f^n\of{x_0}, x_n} \leq \text{diam}\of{U_n} < \epsilon$ for each $n \in \naturals$, it follows $x_0$ $\epsilon$-shadows $\sequenceOf{x_n}:{n \in \naturals}$. Thus, $\p{X, f}$ has the shadowing property, establishing the implication from $\p{3}$ to $\p{1}$.

Now, the equivalence between $\p{1}$ and $\p{3}$ yields an equivalence between $\p{2}$ and $\p{4}$, since $\p{X_\Graph{f}^+, \sigma_\Graph{f}^+}$ is a topological dynamical system, where $\sigma_\Graph{f}^+$ is an isometry.  

Finally, we show $\p{1}$ and $\p{2}$ are equivalent. To this end, suppose $\p{X, f}$ has the shadowing property. Let $\epsilon > 0$. There exists $\delta > 0$ satisfying the shadowing property for $\p{X, f}$, such that every $\delta$-pseudo-orbit is $\epsilon$-shadowed by a point. Suppose $\sequenceOf{\p{f^k\of{x_n}}_{k \in \naturals}}:{n \in \naturals}$ is a $\delta$-pseudo-orbit in $\p{X_{\Graph{f}}^+, \sigma_\Graph{f}^+}$. Then, for each $n \in \naturals$, 
\[
d\of{f\of{x_n}, x_{n+1}} = \sum_{k = 0}^\infinity \Fraction d\of{f\of{x_n}, x_{n+1}} / {2^{k+1}}  = \sum_{k = 0}^\infinity \Fraction d\of{f^{k+1}\of{x_n}, f^k\of{x_{n+1}}} / {2^{k+1}} \leq \delta.
\]
Hence, $\sequenceOf{x_n}:{n \in \naturals}$ is a $\delta$-pseudo-orbit in $\p{X, f}$. There exists $y \in X$ which $\epsilon$-shadows $\sequenceOf{x_n}:{n \in \naturals}$, that is, $d\of{x_n, f^n\of{y}} < \epsilon$ for each $n \in \naturals$. Since 
\[
d\of{f^k\of{x_n}, f^{n+k}\of{y}} = d\of{x_n, f^n\of{y}} < \epsilon
\]
for each $n, k \in \naturals$, it follows for each $n \in \naturals$,
\[
\sum_{k = 0}^\infinity \Fraction d\of{f^k\of{x_n}, f^{n+k}\of{y}} / {2^{k+1}} < \epsilon. 
\]
Thus, $\sequenceOf{f^k\of{y}}:{k \in \naturals}$ $\epsilon$-shadows $\sequenceOf{\p{f^k\of{x_n}}_{k \in \naturals}}:{n \in \naturals}$. Hence, $\p{X_\Graph{f}^+, \sigma_\Graph{f}^+}$ has the shadowing property, yielding the implication from $\p{1}$ to $\p{2}$. 

Conversely, suppose $\p{X_\Graph{f}^+, \sigma_\Graph{f}^+}$ has the shadowing property.  Let $\epsilon > 0$. There exists $\delta > 0$ satisfying the shadowing property for $\p{X_\Graph{f}^+, \sigma_\Graph{f}^+}$, such that every $\delta$-pseudo-orbit is $\Fraction \epsilon / 2$-shadowed by a point. Suppose $\sequenceOf{x_n}:{n \in \naturals}$ is a $\delta$-pseudo-orbit in $\p{X, f}$. Then, for each $n \in \naturals$, $d\of{f\of{x_n}, x_{n+1}} \leq \delta$. Observe $\sequenceOf{\p{f^k\of{x_n}}_{k \in \naturals}}:{n \in \naturals}$ is a $\delta$-pseudo-orbit in $\p{X_\Graph{f}^+, \sigma_\Graph{f}^+}$, and thus $\Fraction \epsilon / 2$-shadowed by some $\sequenceOf{f^k\of{y}}:{k \in \naturals} \in X_\Graph{f}^+$. Therefore, for each $n \in \naturals$, 
\[
\sum_{k = 0}^\infinity \Fraction d\of{f^k\of{x_n}, f^{n + k}\of{y}} / {2^{k+1}} < \Fraction \epsilon / 2, 
\]
implying $d\of{x_n, f^n\of{y}} < \epsilon$. Hence, $y$ $\epsilon$-shadows $\sequenceOf{x_n}:{n \in \naturals}$. Thus, it follows $\p{X, f}$ has the shadowing property, and we are done. 
\end{proof}

The equivalence between $\p{1}$ and $\p{3}$ in Proposition~\ref{proposition:isometry-mahavier} generalises below, similar to how~\cite[Theorem~3]{meddaugh_shadowing_recurrence} generalises to Proposition~\ref{proposition:diagonal-totally-disconnected}.  Proposition~\ref{prop:id-shadowing-diagonal-totally_disconnected-counterexamples} shows that it cannot be improved upon. 

\begin{proposition}\label{proposition:isometry-totaldisconnected}
Suppose $\p{X, G}$ is a CR-dynamical system, such that there exists an isometry $f \colon X \to X$ with $\Graph{f} \subseteq G$. Then,
\begin{enumerate}
\item if $\p{X, G}$ has the $\p{2, 1}$-shadowing property, then $X$ is totally disconnected; 
\item if $X$ is totally disconnected, then $\p{X, G}$ has the $\p{1, 2}$-shadowing property. 
\end{enumerate}
\end{proposition}
\begin{proof}
Suppose $\p{X, G}$ has the $\p{2, 1}$-shadowing property. By Observation~\ref{observation:contraction}, $\p{X, \Graph{f}}$ has the $\p{2, 1}$-shadowing property, implying $\p{X, f}$ has the shadowing property. Thus, $X$ is totally disconnected by Proposition~\ref{proposition:isometry-mahavier}, establishing $\p{1}$.

Suppose $X$ is totally disconnected. By Proposition~\ref{proposition:isometry-mahavier}, $\p{X, f}$ has the shadowing property, implying $\p{X, \Graph{f}}$ has the $\p{1, 2}$-shadowing property. By Observation~\ref{observation:extension}, it follows $\p{X, G}$ has the $\p{1, 2}$-shadowing property, yielding $\p{2}$. 
\end{proof}

We obtain the following generalisation of Proposition~\ref{prop:21-implies-diagonal}. 

\begin{proposition}\label{prop:21-isometry-implies-single-val}
Suppose $\p{X, G}$ is a CR-dynamical system, such that there exists an isometry $f : X \to X$ with $\Graph{f} \subseteq G$. If $\p{X, G}$ has the $\p{2, 1}$-shadowing property, then $G = \Graph{f}$. 
\end{proposition}
\begin{proof}
To derive a contradiction, suppose there exists $y \neq z \in G\of{x}$ for some $x \in X$. Let $\epsilon > 0$. There exists $\delta > 0$ satisfying $\p{2, 1}$-shadowing property on $\p{X, G}$. Observe $\sequence{x, y, f\of{y}, f^2\of{y}, \ldots}$ and $\sequence{x, z, f\of{z}, f^2\of{z}, \ldots}$ are $\p{\delta, 2}$-pseudo-orbits in $\p{X, G}$. Hence, there exists $u, v \in X$ which $\p{\epsilon, 1}$-shadow $\sequence{x, y, f\of{y}, f^2\of{y}, \ldots}$ and $\sequence{x, z, f\of{z}, f^2\of{z}, \ldots}$, respectively.  In particular, $d\of{x, u} < \epsilon$ and $d\of{y, f\of{u}} < \epsilon$, implying 
\[
d\of{f\of{x}, y} \leq  d\of{f\of{x}, f\of{u}} + d\of{y, f\of{u}}  =  d\of{x, u} + d\of{y, f\of{u}} < 2\epsilon.
\]
Similarly, one obtains $d\of{f\of{x}, z} < 2\epsilon$. However, this implies 
\[
d\of{y, z} \leq d\of{f\of{x}, y} + d\of{f\of{x}, z} < 4\epsilon. 
\]
Since $\epsilon > 0$ is arbitrary, we must have $y = z$, which is a contradiction. Thus, $G\of{x}$ is a singleton for each $x \in X$, forcing $G = \Graph{f}$ (as desired).  
\end{proof}

Propositions~\ref{proposition:isometry-totaldisconnected} and~\ref{prop:21-isometry-implies-single-val} together establish $\p{2}$ and $\p{3}$ of Theorem~\ref{theorem:main-C}. We now turn our attention to $\p{4}$ of Theorem~\ref{theorem:main-C}. Note that for every CR-dynamical system $\p{X, G}$, $X_G^+$ is totally disconnected if $X$ is totally disconnected. 

\begin{proposition}\label{prop:totally_disconnected-implies-Mahavier}
Let $\p{X, G}$ be a CR-dynamical system. If $X$ is totally disconnected, then $X_G^+$ is totally disconnected. 
\end{proposition}
\begin{proof}
Observe $X_G^+$ is a closed subspace of $\prod_{n = 0}^\infinity X$, where the latter is totally disconnected (being the product of totally disconnected space $X$).  It follows $X_G^+$ is totally disconnected. 
\end{proof}

The converse does not hold for every CR-dynamical system.  For consider the CR-dynamical system $\p{X, G}$, where $X = [0, 1]$ and $G = \set{\p{0, 0}, \p{0, 1}, \p{1, 0}, \p{1, 1}}$. Clearly, $X$ is not totally disconnected. However, $X_G^+$ is homeomorphic to the Cantor space, and is therefore totally disconnected. We now prove $\p{4}$ of Theorem~\ref{theorem:main-C}. 

\begin{proposition}\label{proposition:graph-contains-isometry-then-X-totally-disconnected-iff-Mahavier-totally-disconnected}
Let $\p{X, G}$ be a CR-dynamical system, such that there exists an isometry $f : X \to X$ with $\Graph{f} \subseteq G$. Then, $X$ is totally disconnected if, and only if, $X_G^+$ is totally disconnected.
\end{proposition}
\begin{proof}
The forward direction is proved in Proposition~\ref{prop:totally_disconnected-implies-Mahavier}. Conversely, suppose $X_G^+$ is totally disconnected. To derive a contradiction, suppose $X$ is not totally disconnected. Then, there is a non-trivial connected component $C$. Choose $\delta > 0$ to be less than the diameter of $C$, and $x \neq y \in C$ such that $d\of{x, y} > \delta$. As $X_G^+$ is a totally disconnected compact metric space, there is a finite cover $\collection{U}$ which consists of pairwise disjoint clopen sets, each with diameter less than our chosen $\delta$. Now, set $\eta := \min\setOf{d\of{U, V}}:{U \neq V \in \collection{U}}$, and let $\epsilon \in \p{0, \eta}$ be arbitrary. By~\cite[Exercise $4.23$ Part $\p{d}$]{nadler}, since $C$ is connected, there is a sequence $\sequenceOf{x_n}:{n \in \naturals}$ and $m > 0$ such that $x_0 = x$, $x_k = y$ for each $k \geq m$, and $d\of{x_n, x_{n+1}} < \epsilon$ for each $n \in \naturals$. For each $n \in \naturals$, consider $\sequenceOf{f^k\of{x_n}}:{k \in \naturals} \in X_G^+$. Indeed, for each $n \in \naturals$,
\[
\sum_{k = 0}^\infinity \Fraction d\of{f^k\of{x_n}, f^k\of{x_{n+1}}} / {2^{k+1}} = \sum_{k=0}^\infinity \Fraction d\of{x_n, x_{n+1}} / {2^{k+1}} < \sum_{k=0}^\infinity \Fraction \epsilon / {2^{k+1}} = \epsilon. 
\]
It follows there exists $U \in \collection{U}$ such that $\sequenceOf{f^k\of{x_n}}:{k \in \naturals} \in U$ for each $n \in \naturals$.  However, the diameter of $U$ is less than $\delta$, implying 
\[
\delta < d\of{x, y} = \sum_{k = 0}^\infinity \Fraction d\of{x, y} / {2^{k+1}} = \sum_{k = 0}^\infinity \Fraction d\of{f^k\of{x}, f^k\of{y}} / {2^{k+1}} = \sum_{k = 0}^\infinity \Fraction d\of{f^k\of{x_0}, f^k\of{x_m}} / {2^{k+1}} \leq \delta,
\]
a contradiction.  Thus, $X$ is totally disconnected.  
\end{proof}

With Theorem~\ref{theorem:main-C} proved, we turn our attention to Theorem~\ref{theorem:main-D}.  We first require a lemma.  

%We consider how $\p{i, j}$-shadowing in $\p{X, G}$ relates to shadowing in the Mahavier dynamical system $\p{X_G^+, \sigma_G^+}$.  

For each $\delta \geq 0$, let
\[
G[\delta] := \setOf{\p{x,y} \in X \times X}:{\exists \p{a,b}\in G
\text{ such that }
d\of{x,a}\leq\delta
\text{ and }
d\of{y,b}\leq\delta},
\]
and let
\[
P_\delta := \setOf{\sequenceOf{x_k}:{k\in\naturals}\in X^{\naturals}}:{\p{x_k,x_{k+1}}\in G[\delta]
\text{ for every }k\in\naturals}.
\]

\begin{lemma}\label{lemma:approximate-paths}
For every $\eta>0$, there exists $\delta>0$ such that, for every
$\mathbf{x}\in P_\delta$, there exists
$\mathbf{a}\in X_G^+$ satisfying
\[
\rho\of{\mathbf{x},\mathbf{a}}<\eta.
\]
The same conclusion holds for every tail of every
$\mathbf{x}\in P_\delta$.
\end{lemma}
\begin{proof}
For each $\delta \geq 0$, observe $G[\delta]$ is closed in $X \times X$.  For it is the projection of 
\[
\setOf{\p{x,y,a,b}\in X^4}:
{\p{a,b}\in G,\ 
d\of{x,a}\leq\delta,\ 
d\of{y,b}\leq\delta}
\]
of the compact space $X^4$. Consequently, $P_\delta$ is a closed subset of the compact space $X^{\naturals}$ and is therefore compact. Notice as well
\[
\bigcap_{\delta>0}P_\delta=X_G^+.
\]
Indeed, if $\mathbf{x}$ belongs to every $P_\delta$, then each
$\p{x_k,x_{k+1}}$ belongs to every $G[\delta]$. Since $G$ is
closed, this implies
\[
\p{x_k,x_{k+1}}\in G
\]
for every $k$, and hence $\mathbf{x}\in X_G^+$.

Suppose the conclusion fails for some $\eta>0$. Then, for each
positive integer $m$, there exists
$\mathbf{x}^{(m)}\in P_{1/m}$ such that
\[
\rho\of{\mathbf{x}^{(m)},\mathbf{a}}\geq\eta
\]
for every $\mathbf{a}\in X_G^+$. By compactness of
$X^{\naturals}$, there exists a subsequence of
$\sequenceOf{\mathbf{x}^{(m)}}:{m\in\naturals}$ which converges to a
point $\mathbf{x}\in X^{\naturals}$. For every positive integer
$r$, the tail of this subsequence eventually lies in the closed
set $P_{1/r}$. Therefore,
\[
\mathbf{x}\in\bigcap_{r\geq1}P_{1/r}=X_G^+,
\]
contradicting
\[
\rho\of{\mathbf{x}^{(m)},X_G^+}\geq\eta.
\]
Finally, every tail of an element of $P_\delta$ also belongs to
$P_\delta$, so the same $\delta$ applies to every tail.
\end{proof}

We now prove Theorem~\ref{theorem:main-D}. 

\begin{proof}[Proof of Theorem~\ref{theorem:main-D}]
Suppose $\p{X, G}$ has the $\p{2, 2}$-shadowing property, and that $\epsilon > 0$. There exists $\delta \in \p{0, \Fraction \epsilon / 2}$ such that every $\p{\delta, 2}$-pseudo-orbit of $\p{X, G}$ is $\p{\Fraction \epsilon / 2, 2}$-shadowed by a point in $X$. Let $\sequenceOf{\mathbf{x}^{(n)}}:{n\in\naturals}$ be a $\Fraction \delta / 2$-pseudo-orbit of $\p{X_G^+,\sigma_G^+}$, where $\mathbf{x}^{(n)} = \sequenceOf{x_k^{(n)}}:{k\in\naturals}$. We have
\[
\rho\of{\sigma_G^+\of{\mathbf{x}^{(n)}}, \mathbf{x}^{(n+1)}} \leq \Fraction \delta / 2
\]
for every $n \in \naturals$.  For each $n \in \naturals$, define 
\[
x_n = x_0^{\p{n}}. 
\]
Since $\mathbf{x}^{\p{n}} \in X_G^+$, we have $x_1^{\p{n}} \in G\of{x_n}$ for every $n \in \naturals$. Hence, 
\[
d\of{x_1^{(n)},x_{n+1}}
\leq
2\rho\of{\sigma_G^+\of{\mathbf{x}^{(n)}},
          \mathbf{x}^{(n+1)}}
\leq \delta, 
\]
implying $\sequenceOf{x_n}:{n\in\naturals}$ is a $\p{\delta,2}$-pseudo-orbit of $\p{X,G}$. Choose
\[
\mathbf{y} =
\sequenceOf{y_n}:{n\in\naturals}
\in X_G^+
\]
which $\p{\Fraction\epsilon/2, 2}$-shadows this pseudo-orbit. Therefore,
\[
d\of{x_n,y_n}<\Fraction\epsilon/2
\]
for every $n\in\naturals$. For each $n\in\naturals$, let 
\[
\mathbf{x}_{n} = \sequenceOf{x_{n+k}}:{k\in\naturals}. 
\]
We claim
\[
\rho\of{\mathbf{x}^{(n)},\mathbf{x}_{n}} \leq \Fraction  \delta / 2.
\]
Indeed, for every positive integer $k$,
\[
d\of{x_k^{(n)},x_{n+k}} \leq \sum_{j=0}^{k-1}
d\of{x_{k-j}^{(n+j)},x_{k-j-1}^{(n+j+1)}}.
\]
Consequently,
\begin{align*}
\rho\of{\mathbf{x}^{(n)},\mathbf{x}_{n}}
&\leq
\sum_{k=1}^{\infty}
\Fraction{1}/{2^{k+1}}
\sum_{j=0}^{k-1}
d\of{x_{k-j}^{(n+j)},x_{k-j-1}^{(n+j+1)}}\\
&=
\sum_{j=0}^{\infty}
\Fraction{1}/{2^{j+1}}
\rho\of{\sigma_G^+\of{\mathbf{x}^{(n+j)}},
         \mathbf{x}^{(n+j+1)}}\\
&\leq \Fraction \delta / 2,
\end{align*}
Since $d\of{x_{n+k},y_{n+k}}<\Fraction\epsilon/2$ for every
$n, k\in \naturals$, we have
\[
\rho\of{\mathbf{x}_{n},
         \p{\sigma_G^+}^n\of{\mathbf{y}}}
=
\sum_{k=0}^{\infty}
\Fraction{d\of{x_{n+k},y_{n+k}}}/{2^{k+1}}
<
\Fraction\epsilon/2.
\]
It follows that
\[
\rho\of{\mathbf{x}^{(n)},
         \p{\sigma_G^+}^n\of{\mathbf{y}}}
<
\Fraction \delta / 2 +\Fraction\epsilon/2
<
\epsilon
\]
for every $n\in\naturals$. Hence, $\mathbf{y}$ $\epsilon$-shadows
the pseudo-orbit $\sequenceOf{\mathbf{x}^{(n)}}:{n\in\naturals}$, proving that
$\p{X_G^+,\sigma_G^+}$ has the shadowing property.

Suppose $\p{X_G^+,\sigma_G^+}$ has the shadowing property, and that $\epsilon > 0$. There exists $\beta \in \p{0, \Fraction \epsilon / 4}$ such that every $\beta$-pseudo-orbit of $\p{X_G^+,\sigma_G^+}$ is
$\Fraction\epsilon/4$-shadowed. Choose
\[
0<\eta<
\min\set{\Fraction\beta/3,\Fraction\epsilon/4},
\]
and let $\delta>0$ be given by
Lemma~\ref{lemma:approximate-paths} for this choice of $\eta$.

Let
\[
\mathbf{x}
=
\sequenceOf{x_n}:{n\in\naturals}
\]
be a $\p{\delta,2}$-pseudo-orbit of $\p{X,G}$. For every
$n\in\naturals$, there exists $z_n\in G(x_n)$ such that
\[
d\of{z_n,x_{n+1}}\leq\delta.
\]
Thus,
\[
\p{x_n,x_{n+1}}\in G[\delta]
\]
for every $n \in \naturals$, and hence $\mathbf{x}\in P_\delta$. Every tail
\[
\mathbf{x}_n
=
\sequenceOf{x_{n+k}}:{k\in\naturals}
\]
also belongs to $P_\delta$. By
Lemma~\ref{lemma:approximate-paths}, for each $n\in\naturals$,
there exists 
\[
\mathbf{a}_n = \sequenceOf{a_{n, k}}:{k \in \naturals} \in X_G^+
\]
such that $\rho\of{\mathbf{a}_n,\mathbf{x}_{n}}<\eta$.  Since
\[
\sigma\of{\mathbf{x}_n}=\mathbf{x}_{n+1},
\]
we have
\begin{align*}
\rho\of{\sigma_G^+\of{\mathbf{a}_{n}},
         \mathbf{a}_{n+1}}
&\leq
\rho\of{\sigma_G^+\of{\mathbf{a}_{n}},
         \sigma\of{\mathbf{x}_{n}}}
+
\rho\of{\sigma\of{\mathbf{x}_n}, \mathbf{x}_{n+1}}
+
\rho\of{\mathbf{x}_{n+1},\mathbf{a}_{n+1}}\\
&\leq
2\rho\of{\mathbf{a}_{n},\mathbf{x}_{n}}
+
\rho\of{\mathbf{x}_{n+1},\mathbf{a}_{n+1}}\\
&<3\eta<\beta,
\end{align*}
which implies $\sequenceOf{\mathbf{a}_{n}}:{n\in\naturals}$
is a $\beta$-pseudo-orbit of
$\p{X_G^+,\sigma_G^+}$. There exists 
\[
\mathbf{y}
=
\sequenceOf{y_n}:{n\in\naturals}
\in X_G^+
\]
which $\Fraction\epsilon/4$-shadows it. For each
$n\in\naturals$, 
\begin{align*}
d\of{y_n,x_n}
&\leq
d\of{y_n,a_{n, 0}}
+
d\of{a_{n, 0},x_n}\\
&\leq
2\rho\of{\p{\sigma_G^+}^n\of{\mathbf{y}},
          \mathbf{a}_{n}}
+
2\rho\of{\mathbf{a}_{n},\mathbf{x}_{n}}\\
&<
\Fraction\epsilon/2+2\eta
<
\epsilon,
\end{align*}
and thus $y_0$ $\p{\epsilon, 2}$-shadows
$\sequenceOf{x_n}:{n\in\naturals}$. Hence, $\p{X,G}$ has the
$\p{2,2}$-shadowing property, and we are done. 
\end{proof}

%----------------------------------------------------------------
\section*{Declarations}

\begin{acknowledgements}
ChatGPT (OpenAI) was used during the revision of this manuscript to assist
with exposition, organisation, and the checking of mathematical arguments.
The author reviewed and verified all resulting material and takes full
responsibility for the content of the manuscript.
\end{acknowledgements}

%----------------------------------------------------------------
\bibliographystyle{plain}
\bibliography{Main}

%---------------------------------------------------------------
\end{document}

%% file: Definitions.tex
\usepackage{amsmath}
\usepackage{amsfonts}
\usepackage{amsthm}
\usepackage{amssymb}
\usepackage{mathrsfs}

\def\parentheses#1{{\left( {#1} \right)}}
\def\of{\parentheses}

\def\Fraction#1/#2{\frac{#1}{#2}}
\def\fraction#1/#2{\tfrac{#1}{#2}}
\def\solidus#1/#2{%
    {{#1} \!\left/\!\vphantom{{#1}{#2}}\!\right. {#2}}%
}

\def\operator{\operatorname}

\def\script{\mathscr}

\def\blackboard{\mathbb}

\def\infinity{\infty}

\def\naturals{\blackboard N}
\def\integers{\blackboard Z}

\def\set#1{{\left\lbrace {#1} \right\rbrace}}
\def\setOf#1:#2{%
    \set{{#1} \,\left\vert\vphantom{{#1}{#2}}\right. {#2}}%
}

\def\emptySet{\varnothing}

\def\sequence#1{{\left\langle {#1} \right\rangle}}
\def\sequenceOf#1:#2{%
    \sequence{{#1} \,\left\vert\vphantom{{#1}{#2}}\right. {#2}}%
}

\def\familyOf#1:#2{{\parentheses{#1}_{#2}}}

\def\union{\cup}
\def\Union{\bigcup}
\def\intersect{\cap}

\def\functionSpace#1->#2{{{#2}^{#1}}}

\def\function#1:#2->#3{{{#1} \colon {#2} \to {#3}}}
\def\mapsTo{\mapsto}
\def\map#1->#2{\parentheses{{#1} \mapsTo {#2}}}

\def\domain#1{{\operator{dom}\of{#1}}}

\def\Restriction#1|#2{{\left.\! {#1} \right\vert_{#2}}} % \restriction is \upharpoonright.

\def\presentation#1{{\left\langle {#1} \right\rangle}}
\def\presentationOf#1:#2{%
    \presentation{{#1} \,\left\vert\vphantom{{#1}{#2}}\right. {#2}}%
}

\def\collection{\script}

\def\interval#1#2,#3#4{{\left#1 {#2}, {#3} \right#4}}

\def\mahavier#1#2{{\bigstar^{#1}_{i=0} {#2}}}
\def\Graph#1{{\Gamma\of{#1}}}

\def\p{\parentheses}